\documentclass[a4paper,12pt]{amsart}
\usepackage[paper=a4paper,left=25mm,right=25mm,top=25mm,bottom=25mm]{geometry}

\usepackage{amscd}
\usepackage{bbm}
\usepackage{mathtools}
\usepackage{enumitem}

\allowdisplaybreaks[1]

\newtheorem{Theorem}{Theorem}[section]
\newtheorem{Lemma}[Theorem]{Lemma}
\newtheorem{Corollary}[Theorem]{Corollary}

\newtheorem{Remark}[Theorem]{Remark}

\theoremstyle{definition}

\usepackage{amssymb}

\newcommand\C{{\mathbb C}}
\newcommand\R{{\mathbb R}}
\newcommand\X{{\R^d}}
\newcommand\N{{\mathbb N}}
\newcommand\M{{\mathcal M}}

\newcommand\B{{\mathcal B}}

\newcommand\Aff{{\mathrm{Aff} (\X)}}

\newcommand\La{\Lambda}

\newcommand\Ga{\Gamma}
\newcommand\ga{\gamma}

\newcommand{\PC}{{\mathcal P}_{cyl}}
\newcommand{\CF}{\mathcal F}

\newcommand{\D}{\mathcal D}

\DeclareMathOperator{\supp}{supp}

\begin{document}

\pagestyle{plain}
\title{Representations of the Infinite-Dimensional Affine Group}

\date{}
\author{
  \textbf{Yuri Kondratiev}\\
Department of Mathematics, University of Bielefeld, \\
D-33615 Bielefeld, Germany,\\
Dragomanov  University, Kyiv, Ukraine\\
}
\begin{abstract}
We introduce an infinite-dimensional  affine group and construct its irreducible unitary representation. Our approach follows the one used by Vershik, Gelfand and Graev for the diffeomorphism group, but with modifications made necessary by the fact that the group does not act on the phase space. However it is possible to define its action on some classes of functions.
\end{abstract}

\maketitle

\vspace*{3cm}
\vspace{2cm}


\medskip

\newpage
\section{Introduction}

Given a vector space $V$  the affine group  can be described concretely as the semidirect product of $V$ by $\mathrm{GL}(V)$, the general linear group of $V$:
$$
 \mathrm{Aff} (V)=V  \rtimes \mathrm{ GL} (V).
 $$
The action of $\mathrm{GL}(V)$ on $V$ is the natural one (linear transformations are automorphisms), so this defines a semidirect product.

Affine groups play important role in the geometry and its applications, see, e.g., \cite{Ar,Ly}. Several recent papers \cite{AJO,AK,EH,GJ,Jo,Ze} are devoted to representations of the real, complex and $p$-adic affine groups and their generalizations, as well as diverse applications, from wavelets and Toeplitz operators to non-Abelian pseudo-differential operators and $p$-adic quantum groups.

In the particular case of  field $V= \X$ the group $\mathrm{Aff}(\X)$ defined as following.

Consider a function $b:\X \to \X$  which is a step function on $\X$. 
Take another  matrix valued function $A:\X\to L(\X) $ s.t. $A(x)=\mathrm{Id} +A_0(x)$,  $A(x)$ is invertible, $A_0$ is a matrix valued step function
on $\X$.  Introduce an
infinite dimensional affine group
$\Aff (\X)$ that is the set of all pairs $g=(A,b)$ with component satisfying assumptions above. Define the group
operation
$$
g_2 g_1=  (A_2,b_2) (A_1, b_1) = (A_1 A_2, b_1 +A_1 b_2).
$$
The unity in this group is $e=(\mathrm{Id} ,0)$.
For $g\in \Aff(\X)$ holds $g^{-1}= (A^{-1}, -A^{-1}b)$.
It is clear that for step mappings we use these definitions are correct.
Our aim is to construct irreducible representations of $\Aff (\X)$. As a rule, only special classes of irreducible representations can be constructed for infinite-dimensional groups. For various classes of such groups, special tools were invented; see \cite{Is,Ko} and references therein.

We will follow an approach by Vershik-Gefand -Graev \cite{VGG75} proposed in the case of
the group of diffeomorphisms.  A direct application of this approach meets certain difficulties
related with the absence of the possibility to define the action of the group $\Aff (\X)$ on a phase space
similar to \cite{VGG75}.  A method to overcome this problem is the main technical step in the present paper.
We wold like to mention that a similar approach was already used in \cite{PAFF} for the construction
of the representation for p-adic infinie dimensional affine group.

\maketitle

\section{Infinite dimensional   affine group}

In our definitions and studies of vector and matrix  valued functions on $\X$ we
will use as basic functional spaces collections of step mappings. 
It means that each such mapping is a finite sum of indicator functions with measurable
bounded supports with constant vector/matrix coefficients. Such spaces of functions on 
$\X$ are rather unusual in the framework of infinite dimensional groups but we will try
to show that their use is natural for the study of affine groups.

For $x\in\X$ consider the section $G_x= \{g(x)\; |\; g\in \Aff(\X)\}$.
It is an affine group with constant coefficients. Note that for a ball $B_N (0) \subset \X$ with the radius $N$
centered at zero we have $g(x)= (1,0), x\in B^c_N(0)$. 

Define the action of $g$ on a point $x\in\X$ as
$$
gx= g(x)x =  A(x)^{-1} (x+b(x)).
$$
Denote the orbit $O_x=\{gx| g\in G_x\}\subset \X$.
Actually,  as a set $O_x=\X$ but elements of this set are
parametrized by $g\in G_x$.
For any element $y\in O_x$  and $h\in G_x$ we can define
$hy= h(gx)= (hg)x\in O_x$. It means that we have the group
$G_x$ action on the orbit $O_x$. 

It gives

$$
(g_1g_2)(x) x= g_1(x)( g_2(x)x)
$$
 that corresponds to the group multiplication
 $$
 g_2 g_1=  (A_2,b_2) (A_1, b_1) = (A_1 A_2, b_1 +A_1 b_2)
 $$
 considered in the given point $x$.

\begin{Remark}
The situation we have is quite different w.r.t. the standard 
group of motions on a phase space. Namely,
we have one fixed point $x\in\X$ and the section group
$G_x$  associated with this point. Then we have the motion of
$x$ under the action of $G_x$. It gives the group action on the 
orbit $O_x$.
\end{Remark}

We will  use  the configuration space $\Ga(\X)$, i.e., the set of all locally finite 
subsets of $\X$. 

Each configuration may be identified with the measure
$$
\gamma(dx) = \sum_{x\in\gamma} \delta_x 
$$
which is a positive  Radon measure on $\X$: $\gamma\in \M(\X)$.
We define the vague topology on $\Ga(\X)$ as the weakest topology for
which all mappings
$$
\Ga(\X) \ni \ga \mapsto <f,\gamma>\in \R,\;\; f\in C_0(\X)
$$
are continuous. The Borel $\sigma$-algebra for this topology denoted
$\B(\Ga(\X))$.

For $\ga\in \Ga(\X)$, $\ga=\{x\}\subset \X$ define
$g\gamma$ as a motion of the measure $\ga$:

$$
g\ga=\sum_{x\gamma} \delta_{g(x)x}\in \M(\X). 
$$
Here we have the group action of $\Aff(\X)$ produced by individual transformations 
of points from the configuration. Again, as above, we move a fixed configuration using
previously defined actions of $G_x$ on $x\in\ga$. 

Note that $g\gamma$ is not more a configuration. More precisely, for some $B_N(0) $
the set $(g\ga)_N= g\ga\cap B_N^c(0)$ is a configuration in $B^c_N(0) $ but the finite part 
of $g\ga$ may include multiple points.

For any $f\in \D(\X,\C)$ we have corresponding cylinder  function on $\Ga(\X)$:
$$
L_f(\ga)=  <f,\ga > = \int_{\X} f(x)\ga(dx) = \sum_{x\in \ga} f(x).
$$
Denote $\PC$ the set of all  cylinder polynomials generated by such functions.
More generally, consider functions of the form

\begin{equation}
\label{cyl}
F(\ga)= \psi(<f_1,\ga>,\dots, <f_n,\ga>),\; \ga\in\Ga(\X),  f_j\in \D(\X), \psi\in C_b(\R^n).
\end{equation}

These functions form the set $\CF_b(\Ga(\X))$ of all  bounded cylinder functions.

For any clopen set $\Lambda \in \mathcal{O}_b(\X)$ (also called a finite volume) denote
$\Ga(\Lambda)$ the set of all (with necessity finite) configurations
in $\La$.   We  have as before the vague topology on this space and 
the Borel $\sigma$-algebra $\B(\Ga(\La))$ is generated by functions
$$
\Ga(\La)\ni\ga \mapsto <f,\ga>\in\R
$$
for $f\in C_0 (\La)$.  For any $\La\in \mathcal{O}_b(\X)$ and $T\in \B(\Ga(\La))$ 
define a cylinder set 
$$
C(T)=\{\ga\in\Ga(\X)\;|\; \ga_{\La}=\ga \cap \La \in T\}.
$$
Such sets form a $\sigma$-algebra $\B_{\La}(\Ga(\X))$ of cylinder sets
for the finite volume $\La$.  The set of bounded functions on $\Ga(\X)$ measurable
w.r.t. $\B_{\La}(\Ga(\X))$  we denote $B_{\La}(\Ga(\X))$. That is a set of cylinder functions
on $\Ga(\X)$.  As a generating family for this set we can use the functions of the form
$$
F(\ga)= \psi(<f_1,\ga>,\dots, <f_n,\ga>),\; \ga\in\Ga(\X),  f_j\in C_0(\La), \psi\in C_b(\R^n).
$$

For so-called one-particle functions $f:\X\to\R, f\in\D(\X)$ consider 
$$
(gf)(x)= f(g(x) x), x\in \X.
$$
 Then $gf\in \D(\X)$. Thus,
 we have the group action 
 $$
 \D(\X)\in f \mapsto gf\in \D(\X),\;\;g\in\Aff
 $$
 of the infinite dimensional group $\Aff$ in the space of functions
 $\D(\X)$.

Note that due to our definition, we have
$$
<f, g\ga> = <gf,\ga>
$$
and it is reasonable to define for cylinder functions (\ref{cyl}) the action of the group $\Aff$
as
$$
(V_g F)(\ga)= \psi(<gf_1,\ga>,\dots <gf_n,\ga>.
$$ 
Obviously $V_g: \CF_b (\Ga(\X))\to \CF_b(\Ga(\X))$.

Denote $m(dx)$ the Haar measure on $\X$. The dual transformation to one-particle motion is defined
via the following relation
$$
\int_{\X} f(g(x)x) m(dx)=\int_{\X} f(x) g^\ast m(dx)
$$
if exists such measure $g^\ast m$ on $\X$.

\begin{Lemma}
\label{gm}

For each $g\in \Aff$ 
$$
g^\ast m(dx)=  \rho_{g}(x)  m(dx)
$$
where $\rho_g  = 1_{B_R^c(0) } + r_g^0,\;\; r_g^0\in \D(\X,\R_+).$
Here as above
$$
B_R^c(0)= \{x\in\X\;|\; |x|_p \geq R\}.
$$

\end{Lemma}

\begin{proof}
We have following representations for coefficients of $g(x)$:

$$
b(x)= \sum_{k=1}^{n} b_k 1_{B_k}(x) ,
$$
$$
a(x)= \sum_{k=1}^{n} a_k 1_{B_k}(x) + 1_{B^c_R(0)}(x)
$$
where $B_k$ are certain balls in $\X$.
Then
$$
\int_{\X} f(g(x)x) m(dx)= \sum_{k=1}^n \int_{B_k} f(\frac{x+b_k}{a_k}) m(dx) + \int_{B^c_R (0)} f(x) m(dx) =
$$
$$
\sum_{k=1}^{n} \int_{C_k} f(y) |a_k|_p m(dy) + \int_{B^c_R(0)} f(y) m(dy),
$$
where 
$$
C_k= a_k^{-1}(B_k + b_k).
$$
Therefore,
$$g^\ast m= (\sum_{k=1}^n  |a_k|_p 1_{C_k} + 1_{B^c_R(0)}) m.
$$
Note that informally we can write
$$
(g^\ast m)(dx) = dm(g^{-1}x).
$$
\end{proof}

Note that by the duality we have the group action on the Lebesgue  measure. Namely,
for $f\in \D(\X)$ and $g_1, g_2\in \Aff$
$$
\int_{\X} (g_2 g_1) f(x)  m(dx)= \int_{\X} g_1 f (x)  (g_2^\ast m) (dx) =
$$
$$
\int_{\X} f(x) (g_1^\ast  g_2^\ast m)(dx)= \int_{\X} f(x) ((g_2 g_1)^\ast m)(dx).
$$
In particular
$$
(g^{-1})^\ast (g^\ast m)= m.
$$

\begin{Lemma} Let $F\in B_\La (\Ga(\X))$ and $g\in\Aff $ has the form
$g(x)=(1, h1_{B}(x))$ with certain $h\in \X$ and $B\in \mathcal{O}_b(\X)$ s.t. $\La\subset B$.
Then
$$
V_gF\in B_{\La -h} (\Ga(\X)).
$$

\end{Lemma}
\begin{proof}
Due to the formula for the action $V_gF$ we need to analyze the support
of functions $f_j (x+h1_B(x))$ for $\supp f_\subset \La$. If $x\in B^c$ then 
$x\in \La^c$ and therefore $f_j (x+h1_B(x))=f_j(x)=0$. For $x\in B$
we have $f_j(x+h)$ and only for $x+h\in \La$ this value may be nonzero,
i.e., $\supp g f_j \subset \La- h$.

\end{proof}

Denote $\pi_m$ the Poisson measure on $\Ga(\X)$ with the intensity measure $m$.

\begin{Lemma}
\label{V}
For all $F \in \PC$ or $F\in \CF_b (\Ga(\X))$ and $g\in \Aff $ holds
$$
\int_{\Ga(\X)} V_g F d\pi_m = \int_{\Ga(\X)} Fd\pi_{g^\ast m} .
$$

\end{Lemma}

\begin{proof}
It is enough to show this equality for exponential functions
$$
F(\ga)= e^{<f,\ga>},\;\; f\in\D(\X).
$$

We have
$$
\int_{\Ga(\X)} V_g F d\pi_m = \int_{\Ga(\X)} e^{<gf, \ga>} d\pi_m(\ga)=
$$
$$
\exp[ \int_{\X} (e^{gf(x)} -1) dm(x)] = \exp[ \int_{\X} (e^{f(x)} -1) d(g^{\ast} m)(x)=
$$
$$
\int_{\Ga(\X)} F d\pi_{g^\ast m }.
$$

\end{proof}

\begin{Remark} For all functions $F,G\in \CF(\Ga(\X))$ a similar 
calculation shows
$$
\int_{\Ga(\X)} V_g F  \; Gd\pi_m = \int_{\Ga(\X)} F  \; V_{g^{-1}} G d\pi_{g^\ast m} .
$$
\end{Remark}
Let $\pi_m$ be the Poisson measure on $\Ga(\X)$ with the intensity
measure $m$. For any $\La\in \mathcal{O}_b(\X)$ consider the distribution $\pi_m^\La$
of $\pi_m$ in $\Ga(\La)$ corresponding the projection $\ga\to \ga_\La$.
It is again a Poisson measure $\pi_{m_\La}$ in $\Ga(\La)$ with the intensity
$m_\La$ which is the restriction of $m$ on $\La$.  Infinite divisibility of 
$\pi_m$ gives for $F_j\in B_{\La_j}(\Ga(\X)), j=1,2$ with $\La_1\cap \La_2=\emptyset$
$$
\int_{\Ga(\X)}  F_1(\ga) F_2(\ga) d\pi_m(\ga)= \int_{\Ga(\X)}  F_1(\ga) d\pi_m(\ga)
\int_{\Ga(\X)}  F_2(\ga) d\pi_m(\ga)=
$$
$$
\int_{\Ga(\La_1)} F_1 d\pi^{\La_1}_m \int_{\Ga(\La_2)} F_2 d\pi^{\La_2}_m.
$$

\begin{Lemma}

For any $F\in B_\La(\Ga(\X)$ and $g=(1, h1_B)\in \Aff $ with $\La \cap (B+h)=\emptyset$ holds
$$
\int_{\Ga(\X)} (V_g F)(\ga) d\pi_m(\ga)= \int_{\Ga(\X)} F(\ga)d\pi_m(\ga).
$$

\end{Lemma}

\begin{proof}
Due to our calculations above we have
$$
\int_{\Ga(\X)} (V_gF)(\ga) d\pi_m(\ga)= \int_{\Ga(\X)} F(\ga) d\pi_{g^{\ast}m}(\ga)=
$$
$$
\int_{\Ga(\La)} F(\eta) d\pi^{\La}_{g^{\ast}m} (\eta) =\int_{\Ga(\La)} F(\eta) d\pi_{ (g^{\ast}m)_\La} (\eta).
$$
But we have shown 
$$
(g^{\ast}m)(dx)= (1+ 1_{B+h}(x)) m(dx) = m(dx)
$$
for $x\in \La$, i.e.,  $(g^{\ast}m)_\La =m$. 

\end{proof}

\begin{Lemma}
\label{prod}
For any $F_1,F_2 \in \CF_b(\Ga(\X))$ there exists $g\in\Aff$ such that
$$
\int_{\Ga(\X)} F_1 \; V_g F_2  d\pi_m = \int_{\Ga(\X)} F_1 d\pi_m  \int_{\Ga(\X)} F_2 d\pi_m .
$$

\end{Lemma}

\begin{proof}
By the definition, $F_j\in B_{\La_j}(\Ga(\X)), j=1,2$ for some $\La_1,\La_2 \in \mathcal{O} (\X)$.

 Let us take $g=(1, h1_B)$ with the following assumptions:
 $$
 \La_2\subset B,\;\; \La_1\cap (\La_2-h) =\emptyset,\;\; \Lambda_2\cap (B+h) =\emptyset.
 $$
 Then accordingly to previous lemmas 
$$
\int_{\Ga(\X)} F_1 V_g F_2 d\pi_m = \int_{\Ga(\X)} F_1 d\pi_m  \int_{\Ga(\X)} F_2 d\pi_m .
$$

\end{proof}

\section{$\Aff$ and Poisson measures}

For $F\in \PC $  or $F\in \CF_b (\Ga(\X))$, we consider the motion of $F$ by $g\in \Aff$ given by
the operator $V_g$.
Operators $V_g$   have the group property
defined point-wisely: for any $\ga \in \Ga(\X) $

$$
(V_h (V_gF))(\ga)= (V_{hg} F) (\ga).
$$
This equality is the consequence of  our definition of the group action of $\Aff$
on cylinder functions.

As above,
consider  $\pi_m$, the Poisson measure on $\Ga(\X)$ with the intensity measure $m$. For the transformation
$V_g$ the dual object is defined as the measure $V^\ast_g \pi_m$ on $\Ga(\X)$ given by the relation
$$
\int_{\Ga(\X)} (V_gF) (\ga) d\pi_m(\ga) =\int_{\Ga(\X)} F(\ga)  d(V^\ast_g \pi_m)(\ga),
$$
where $V^\ast_g \pi_m= \pi_{g^\ast m}$, see Lemma \ref{V}.

\begin{Corollary}
For any $g\in \Aff$ the  Poisson measure $V_g^\ast \pi_m$ is absolutely continuous
 w.r.t. $\pi_m$  with the Radon-Nykodim derivative
$$
R(g,\ga)= \frac{d\pi_{g^\ast m}(\ga)}{d\pi_{ m} (\ga)} \in L^1(\pi_m).
$$.

\end{Corollary}

\begin{proof}
Note that density $\rho_g  = 1_{B_R^c(0) } + r_g^0,\;\; r_g^0\in \D(\X,\R_+)$ of $g^\ast m$ w.r.t. $m$
may be equal zero on some part of $\X$ and, therefore, the equivalence of of considered
Poisson measures is absent. Due to \cite{LS03}, the Radon-Nykodim derivative
$$
R(g,\ga)= \frac{d\pi_{g^\ast m}(\ga)}{d\pi_{ m} (\ga)}
$$
exists if
$$
\int_{\X} |\rho_g(x)-1| m(dx)= \int_{B_R(0)} |1-r_g^0 (x)| m(dx) <\infty.
$$
\end{proof}

\begin{Remark}
As in the proof of Proposition 2.2 from \cite{AKR} we have an explicit formula for $R(g,\ga)$:

$$
R(g,\ga)= \prod_{x\in\ga} \rho_g (x) \exp(\int_{\X} (1-\rho_g(x)) m(dx).
$$
The point-wise existence of this expression is obvious.

\end{Remark}

This fact gives us the possibility to apply the Vershik-Gelfand-Graev approach realized by these authors for the case
of diffeomorphism group.

Namely, for $F\in \PC$ or $F\in \PC(\Ga(\X)$ and $g\in \Aff$ introduce operators
$$
(U_g F)(\ga) = (R(g^{-1} ,\ga) )^{1/2} (V_gF)(\ga).
$$

\begin{Theorem}

Operators   $U_g,\; g\in \Aff$ are unitary in $L^2 (\Ga(\X), \pi_m)$ and give an irreducible representation
of $\Aff$.

\end{Theorem}

\begin{proof}
Let us check the isometry property  of these operators.
We have using Lemmas \ref{V}, \ref{gm}
$$
\int_{\Ga(\X)} |U_g|^2 d\pi_m = \int_{\Ga(\X)} |V_g F|^2(\ga) d\pi_{(g^{-1})^\ast m} (\ga)=
$$
$$
\int_{\Ga(\X)} |F(\ga)|^2 d\pi_{(gg^{-1})\ast m}(\ga)= \int_{\Ga(\X)} |F(\ga)|^2 d\pi_{ m}(\ga).
$$
From Lemma \ref{V} follows that $U_g^\ast = U_{g^{-1}}.$

We need only to check irreducibility that shall
follow from the ergodicity of Poisson measures \cite{VGG75}. But to this end we need first of all to define the action of
 the group $\Aff$ on sets from $\B(\Ga(\X)$.  As we pointed out above, we can not define this
 action point-wisely. But we can define the action of     operators $V_g$ on the indicators $1_A(\ga)$ for
 $A\in \B(\Ga(Q))$. Namely, for given $A$ we take a sequence of cylinder sets $A_n, n\in \N$ such that
 $$
 \pi_{m}(A\Delta A_n) \to 0, n\to \infty.
 $$
 Then
 $$
 U_g 1_{A_n} =V_g 1_{A_n} (R(g^{-1} ,\cdot) )^{1/2}  \to G (R(g^{-1} ,\cdot) )^{1/2}  \in L^2(\pi_m), n\to\infty
 $$
 in $L^2(\pi_m)$. Each $V_g 1_{A_n} $ is an indicator of a cylinder set and
 $$
 V_g 1_{A_n}  \to G \;\; \pi_m - a.s., n\to \infty.
 $$
 Therefore,
 $G=1$ or $G=0$ $\pi_m$-a.s. We denote this function $V_g 1_A$.

 For the proof of the ergodicity of the measure $\pi_m$ w.r.t. $\Aff$ we need to show the following fact:
 for any $A\in \B(\Ga(\X))$ such that $\forall g\in\Aff\;\; V_g 1_A = 1_A\; \pi_m- a.s.$ holds $\pi_m(A)= 0$
 or $\pi_m(A)= 1$.

 Fist of all, we will show that for any pair of sets $A_1, A_2 \in \B(\Ga(Q))$ with $\pi_m(A_1)>0,\;\;
 \pi_m(A_2) >0$ there exists $g\in\Aff$ such that
 \begin{equation}
 \label{ineq}
 \int_{\Ga(\X)} 1_{A_1} V_g 1_{A_2} d\pi_m \geq \frac{1}{2} \pi_m(A_1) \pi_m(A_2).
 \end{equation}
 Because any Borel set may be approximated by cylinder sets, it is enough to show this fact
 for cylinder sets. But for such sets due to Lemma \ref{prod}  we can choose $g\in \Aff$ such that
$$
 \int_{\Ga(\X)} 1_{A_1} V_g 1_{A_2} d\pi_m =  \pi_m(A_1) \pi_m(A_2).
 $$
Then using an approximation we will have (\ref{ineq}).

 To finish the proof of the ergodicity, we consider any $A\in\B(\Ga(\X)$  such that
 $$
 \forall g\in \Aff\; V_g1_A = 1_A \;\;\pi_m - a.s.,\;\; \pi_m(A)>0.
 $$
 We will show that then $\pi_m(A)= 1$.  Assume $\pi_m(\Ga\setminus A) >0$.
 Due to the statement above, there exists $g\in \Aff$ such that
 $$
 \int_{\Ga(\X)} 1_{\Ga\setminus A} V_g 1_A >0.
 $$
 But due to the invariance of $1_A$  it means
 $$
 \int_{\Ga(\X)} 1_{\Ga\setminus A} 1_A d\pi_m >0
 $$
 that is impossible.
\end{proof}


\begin{thebibliography}{999}
\bibitem{AJO}
H. Airault, S. Jendoubi, and H. Ouerdiane, Unitarising measures for the representations of affine group and associated invariant operators. {\it Bull. Sci. Math.} {\bf 137} (2013), 775--790.
\bibitem{AKR}
S.~Albeverio, {Yu}.~G. Kondratiev, and M.~R{\"o}ckner. Analysis and geometry on configuration spaces. {\it J.~Funct.~Anal.}, {\bf 154} (1998), 444--500.
\bibitem{AK}
S. Albeverio and S. V. Kozyrev, Frames of $p$-adic wavelets and orbits of the affine group. {\it p-Adic Numbers Ultrametric Anal. Appl.} {\bf 1} (2009), no. 1, 18--33.
\bibitem{Ar}
R. Artzy, {\it Linear Geometry}, Addison-Wesley, Reading, 1965.
\bibitem{EH}
A. S. Elmabrok and O. Hutnik, Induced representations of the affine group and intertwining operators: I. Analytical approach.
{\it J. Phys. A} {\bf 45} (2012), no. 24, 244017, 15 pp.
\bibitem{GJ}
V. Gayral and D. Jondreville, Quantization of the affine group of a local field. {\it J. Fractal Geom.} {\bf 6} (2019), 157--204.
\bibitem{Is}
R.S. Ismagilov, {\it Representations of Infinite-Dimensional Groups}. Translations of Mathematical
Monographs 152, American Mathematical Society, Providence, RI, 1996.
\bibitem{Jo}
D. Jondreville, A locally compact quantum group arising from quantization of the affine group of a local field. {\it Lett. Math. Phys.} {\bf 109} (2019), 781--797.
\bibitem{PAFF}  A.Kochubei and Yu.Kondratiev, Representation of infinite-dimensional  p-adic affine group,
to appear in {\it IDAQP} (2020).
\bibitem{Ko}
A. Kosyak, {\it Regular, Quasi-regular and Induced Representations of Infinite-Dimensional Groups}, European Mathematical Society, Z\"urich, 2018.
\bibitem{LS03}
M. A. Lifshits and E. Yu. Shmileva, Poisson measures that are quasi-invariant with respect to multiplicative transformations.  {\it Theory Probab. Appl.} {\bf 46} (2003), 652--666.
\bibitem{Ly}
R. Lyndon, {\it Groups and Geometry}, Cambridge University Press, 1985.
\bibitem{VGG75}
A. M. Vershik, I. M. Gel'fand, M. I. Graev, Representations of the group of diffeomorphisms, {\it Russian Math. Surveys}, {\bf 30}, no.6 (1975), 1--50.
\bibitem{VVZ}
V. S. Vladimirov, I. V. Volovich and E. I. Zelenov, {\it $p$-Adic Analysis and Mathematical Physics}, World Scientific, Singapore, 1994.
\bibitem{Ze}
A. M. Zeitlin, Unitary representations of a loop $ax+b$ group, Wiener measure and $\Gamma$ -function, {\it J. Funct. Anal.} {\bf 263} (2012), 529--548.
\end{thebibliography}
\end{document}